\documentclass[11pt]{article}
\pdfoutput=1

\usepackage[margin=1.3in]{geometry}

\usepackage{graphicx}
\usepackage{amssymb,amsmath}
\usepackage{amsthm,enumerate}
\usepackage{hyperref,xcolor}
\usepackage{mathrsfs}
\usepackage{extarrows}
\usepackage{textcomp}

\makeatletter

\newdimen\bibspace
\makeatother

\numberwithin{equation}{section}

\newtheorem{theorem}{Theorem}[section]
\newtheorem{lemma}[theorem]{Lemma}

\newtheorem{corollary}[theorem]{Corollary}
\newtheorem{remark}[theorem]{Remark}

\def\XXint#1#2#3{{\setbox0=\hbox{$#1{#2#3}{\int}$ }
\vcenter{\hbox{$#2#3$ }}\kern-.6\wd0}}

\begin{document}

\title{Asymptotic expansion at infinity of solutions of Monge-Amp\`ere type equations}

\author{Zixiao Liu,\quad Jiguang Bao\footnote{Supported in part by Natural Science Foundation of China (11871102 and 11631002).}}
\date{\today}

\maketitle

\begin{abstract}
We obtain a quantitative expansion at infinity of solutions for a kind of Monge-Amp\`ere type equations that origin from mean curvature equations of Lagrangian graph $(x,Du(x))$ and refine the previous study on zero mean curvature equations and the Monge-Amp\`ere equations.


 {\textbf{MSC~2020:}}~~ 35J60;~~35B40
\end{abstract}

\section{Introduction}

In 2018, Wang-Huang-Bao \cite{Wang.Chong-paper} studied the second boundary value problem of Lagrangian mean curvature equation of gradient graph $(x,Du(x))$ in $\left(\mathbb{R}^{n} \times \mathbb{R}^{n}, g_{\tau}\right)$, where $Du$ denotes the gradient of scalar function $u$ and
\begin{equation*}
g_{\tau}=\sin \tau \delta_{0}+\cos \tau g_{0}, \quad \tau \in\left[0, \frac{\pi}{2}\right]
\end{equation*}
is the linearly combined metric of standard Euclidean metric
\begin{equation*}
\delta_{0}=\sum_{i=1}^{n} d x_{i} \otimes d x_{i}+\sum_{j=1}^{n} d y_{j} \otimes d y_{j},
\end{equation*}
with the pseudo-Euclidean metric
\begin{equation*}
g_{0}=\sum_{i=1}^{n} d x_{i} \otimes d y_{i}+ \sum_{j=1}^{n} d y_{j} \otimes d x_{j}.
\end{equation*}
They proved that for domain $\Omega\subset\mathbb R^n$, if $u\in C^2(\Omega)$ is a solution of
\begin{equation}\label{Equ-perturb}
  F_{\tau}\left(\lambda\left(D^{2} u\right)\right)=f(x), \quad x \in \Omega,
\end{equation}
then $Df(x)$ is the mean curvature of gradient graph $(x,Du(x))$ in $\left(\mathbb{R}^{n} \times \mathbb{R}^{n}, g_{\tau}\right)$. Previously,  Warren \cite{Warren} proved that when $f(x)\equiv C_0$ for some constants $C_0$, the mean curvature of $(x,Du(x))$ is zero. In \eqref{Equ-perturb},
$f(x)$ is a scalar function with sufficient regularity, $\lambda\left(D^{2} u\right)=\left(\lambda_{1}, \lambda_{2}, \cdots, \lambda_{n}\right)$ are $n$ eigenvalues of Hessian matrix $D^{2} u$ and
$$
F_{\tau}(\lambda):=\left\{
\begin{array}{ccc}
\displaystyle  \frac{1}{n} \sum_{i=1}^{n} \ln \lambda_{i}, & \tau=0,\\
\displaystyle  \frac{\sqrt{a^{2}+1}}{2 b} \sum_{i=1}^{n} \ln \frac{\lambda_{i}+a-b}{\lambda_{i}+a+b},
  & 0<\tau<\frac{\pi}{4},\\
  \displaystyle-\sqrt{2} \sum_{i=1}^{n} \frac{1}{1+\lambda_{i}}, & \tau=\frac{\pi}{4},\\
  \displaystyle\frac{\sqrt{a^{2}+1}}{b} \sum_{i=1}^{n} \arctan \displaystyle\frac{\lambda_{i}+a-b}{\lambda_{i}+a+b}, &
  \frac{\pi}{4}<\tau<\frac{\pi}{2},\\
  \displaystyle\sum_{i=1}^{n} \arctan \lambda_{i}, & \tau=\frac{\pi}{2},\\
\end{array}
\right.
$$
$a=\cot \tau, b=\sqrt{\left|\cot ^{2} \tau-1\right|}$.

If $\tau=0$, then \eqref{Equ-perturb} becomes the Monge-Amp\`ere type equation
\begin{equation}\label{equ-MA}
\operatorname{det} D^{2} u=e^{nf(x)}\quad\text{in }\mathbb R^n.
\end{equation}
For $f(x)$ being a constant $C_0$, there are Bernstein-type results by J\"orgens \cite{Jorgens}, Calabi \cite{Calabi} and Pogorelov \cite{Pogorelov}, which state that any convex classical solution of \eqref{equ-MA}  must be a quadratic polynomial. See Cheng-Yau \cite{ChengandYau}, Caffarelli \cite{7}, Jost-Xin \cite{JostandXin} and Li-Xu-Simon-Jia \cite{AffineMongeAmpere} for different proofs and extensions. For $f(x)-C_0$ having compact support, there are exterior Bernstein-type results by Ferrer-Mart\'{\i}nez-Mil\'{a}n \cite{FMM99} for $n=2$ and Caffarelli-Li \cite{CL}, which state that any convex solution must be asymptotic to quadratic polynomials at infinity (for $n=2$ we need additional $\ln$-term). For $f(x)-C_0$ vanishing at infinity, there are similar asymptotic results by Bao-Li-Zhang \cite{BLZ}. For $f(x)-C_0$ being a periodic function or asymptotically periodic function, there are classification results by Caffarelli-Li \cite{Peroidic_MA}, Teixeira-Zhang \cite{Peroidic_MA2}  etc.

If $\tau=\frac{\pi}{2}$, then \eqref{Equ-perturb} becomes the  Lagrangian mean curvature equation
\begin{equation}\label{equ-spl}
\sum_{i=1}^{n} \arctan \lambda_{i}\left(D^{2} u\right)=f(x)\quad\text{in }\mathbb R^n.
\end{equation}
For $f(x)$ being a constant $C_0$, there are Bernstein-type results by Yuan \cite{Yu.Yuan1,Yu.Yuan2}, which state that any classical solution of \eqref{equ-spl}  and
\begin{equation}\label{equ-cond-spl}
   D^2u\geq \left\{
  \begin{array}{lll}
    -KI, & n\leq 4,\\
    -(\frac{1}{\sqrt 3}+\epsilon(n))I, & n\geq 5,\\
  \end{array}
  \right.\quad\text{or}\quad C_0>\frac{n-2}{2}\pi,
\end{equation}
must be a quadratic polynomial, where  $I$ denote the unit $n\times n$ matrix, $K$ is a constant and $\epsilon(n)$ is a small dimensional constant.
For $f(x)-C_0$ having compact support, there is an exterior Bernstein-type result by Li-Li-Yuan \cite{ExteriorLiouville}, which states that any classical solution of \eqref{equ-spl} with \eqref{equ-cond-spl}  must be asymptotic to quadratic polynomials at infinity (for $n=2$ we need additional $\ln$-term).

For general $\tau\in [0,\frac{\pi}{2}]$, for $f(x)$ being a constant $C_0$, there are Bernstein-type results under suitable semi-convex conditions by Warren \cite{Warren}, which is based on the results of J\"orgens \cite{Jorgens}-Calabi \cite{Calabi}-Pogorelov \cite{Pogorelov}, Flanders \cite{Flanders} and Yuan \cite{Yu.Yuan1,Yu.Yuan2}. For $f(x)-C_0$ having compact support, there are exterior Bernstein-type results when $n\geq 3$ in our earlier work \cite{bao-liu-2020}, which state that any classical solution of \eqref{Equ-perturb} with suitable semi-convex conditions must be asymptotic to quadratic polynomial at infinity. There are also  higher order expansions at infinity, which give the precise gap between exterior maximal/minimal gradient graph and the entire case. Such higher order expansions
problem was considered for the Yamabe equation and $\sigma_k$-Yamabe equation by Han-Li-Li \cite{Han2019-Expansion}, which refines the study by Caffarelli-Gidas-Spruck \cite{CGS}, Korevaar-Mazzeo-Pacard-Schoen \cite{KMPS},  Han-Li-Teixeira \cite{Han-Li-T-Simgak} etc.

In this paper, we obtain asymptotic expansion at infinity of classical solutions of
\begin{equation}\label{Equ-exterior}
  F_{\tau}(\lambda(D^2u))=f(x)\quad\text{in }\mathbb R^n,
\end{equation}
where $n\geq 3$, $\tau\in [0,\frac{\pi}{4}]$ and $f(x)$ is a perturbation of $f(\infty):=\displaystyle\lim _{x \rightarrow \infty} f(x)$ at infinity. This partially refines previous study \cite{BLZ,CL,RemarkMA-2020,ExteriorLiouville,bao-liu-2020} etc.

Our first result considers asymptotic behavior and higher order expansions of general classical solution of \eqref{Equ-exterior}.  Hereinafter, we let $\varphi=O_m(|x|^{-k_1}(\ln|x|)^{k_2})$ with $m\in\mathbb N, k_1,k_2\geq 0$ denote $$|D^k\varphi|=O(|x|^{-k_1-k}(\ln|x|)^{k_2})
\quad\text{as}~|x|\rightarrow+\infty
$$ for all $0\leq k\leq m$. Let $x^T$ denote the transpose of vector $x\in\mathbb R^n$,  $\mathtt{Sym}(n)$ denote the set of symmetric $n\times n$ matrix,
$\mathcal H_k^n$ denote the $k$-order spherical harmonic function space in $\mathbb R^n$,
$DF_{\tau}(\lambda(A))$ denote the matrix with elements being value of partial derivative of $F_{\tau}(\lambda(M))$ w.r.t $M_{ij}$ variable at matrix $A$ and $[k]$ denote the largest natural number no larger than $k$.
\begin{theorem}\label{Thm-firstExpansion}
  Let $u \in C^{2}\left(\mathbb{R}^{n}\right)$ be a classical solution of \eqref{Equ-exterior}, where
  $f\in C^0(\mathbb{R}^n)
  $ is $C^m$ outside a compact subset of $\mathbb{R}^n$ and satisfies \begin{equation}\label{Low-Regular-Condition}
  \limsup _ { | x | \rightarrow \infty } | x | ^ {\zeta+k} | D^k( f ( x ) - f ( \infty ) ) | < \infty,\quad\forall~ k=0,1,2,\cdots,m
  \end{equation}
  for some $\zeta>2$ and $m\geq 2$.
  Suppose either of the following holds
  \begin{enumerate}[(1)]
    \item \label{case-MA} $D^2u>0$ for $\tau=0$;
    \item \label{case-small}
\begin{equation}\label{Condition-QuadraticGrowth-1}
u(x)\leq C(1+|x|^2)\quad\text{and}\quad D^2u>(-a+b)I,\quad\forall~ x\in\mathbb{R}^n
\end{equation}
for some constant $C$, for $\tau\in (0,\frac{\pi}{4})$;
\item \label{case-inverse}
\begin{equation}\label{Condition-QuadraticGrowth-2}
u(x)\leq C(1+|x|^2)\quad\text{and}\quad D^2u>-I,\quad \forall~ x\in\mathbb{R}^n
\end{equation}
for some constant $C$, for $\tau=\frac{\pi}{4}$.
  \end{enumerate}
Then there exist $c\in\mathbb R, b\in\mathbb R^n$ and $ A\in\mathtt{Sym}(n)$ with $F_{\tau}(\lambda(A))=f(\infty)$ such that
  \begin{equation}\label{equ-asym-Behavior}
 u(x)-\left(\frac{1}{2} x^T A x+b x+c\right)=\left\{
  \begin{array}{llll}
    O_{m+1}(|x|^{2-\min\{n,\zeta\}}), & \zeta\not=n,\\
    O_{m+1}(|x|^{2-n}(\ln|x|)), & \zeta=n,\\
  \end{array}
  \right.
\end{equation}
as $|x|\rightarrow+\infty$.
\end{theorem}

\begin{remark}
The matrix $A$ in Theorem \ref{Thm-firstExpansion}  also satisfies $A> 0$ in case \eqref{case-MA}, $A> (-a+b)I$ in case \eqref{case-small} and $A> -I$ in case \eqref{case-inverse} respectively.
\end{remark}

\begin{remark}\label{example}
Notice that in condition \eqref{Low-Regular-Condition}, we only require $m\geq 2$, which is an improvement to the results for  $m\geq 3$ by Bao-Li-Zhang \cite{BLZ}. It would be an interesting to determin sharp lower bounds for $m$ in Theorem \ref{Thm-firstExpansion}. There has been an example in \cite{BLZ} that shows the decay rate assumption $\zeta>2$ in \eqref{Low-Regular-Condition} is optimal.
\end{remark}

We also have the following higher order expansions for $\zeta>n$, which gives a finer characteristic of the error term in \eqref{equ-asym-Behavior}.
\begin{theorem}\label{Thm-secondExpansion}
  Under conditions of Theorem \ref{Thm-firstExpansion},  there exist $c_0\in\mathbb R$,
  $c_k(\theta)\in\mathcal H_k^n$ with $k=1,2,\cdots,n-[2n-\zeta]-1$ such that
  \begin{equation}\label{equ-asym-expan-1}
  \begin{array}{llll}
  &\displaystyle u ( x ) - \left( \frac { 1 } { 2 } x ^ TA x + bx + c \right)\\
  -&\displaystyle
  c_0(x^T(DF_{\tau}(\lambda(A)))^{-1}x)^{\frac{2-n}{2}}
  -\sum_{k=1}^{n-[2n-\zeta]-1}c_k(\theta)\left(x^T(DF_{\tau}(\lambda(A)))^{-1} x\right)^{\frac{2-n-k}{2}}
  \\=&
  \left\{
  \begin{array}{lllll}
  O_m(|x|^{2-\min\{2n,\zeta\}}), & \min\{2n,\zeta\}-n\not\in\mathbb N,\\
    O_m(|x|^{2-\min\{2n,\zeta\}}(\ln |x|)), & \min\{2n,\zeta\}-n\in\mathbb N,\\
  \end{array}
  \right.
  \end{array}
  \end{equation}
  as $|x|\rightarrow+\infty$, where
  \begin{equation*}
\theta=\frac{(DF_{\tau}(\lambda(A)))^{-\frac{1}{2}}x}{\left(x^T(
DF_{\tau}(\lambda(A)))^{-1}x\right)^{\frac{1}{2}}}.
\end{equation*}
\end{theorem}

\begin{remark}\label{thm-radial}
  By computing $F_{\tau}(\lambda(D^2u))$ of radially symmetric  $u$ of form $\frac{C_1}{2}|x|^2+C_2|x|^{-k}$, we find   expansions \eqref{equ-asym-Behavior} and    \eqref{equ-asym-expan-1} are optimal for all $\zeta>2$ in the sense that
  the series of  $k$ doesn't exists or cannot be taken up to  $n-[2n-\zeta]$
  when $2<\zeta\leq n$ or $\zeta>n$ respectively
  since $c_{n-[2n-\zeta]}$ doesn't belong to space $\mathcal H_n^{n-[2n-\zeta]}$ in general.
\end{remark}

The paper is organized as follows.  In section \ref{sec-convergeHessian} we prove that the Hessian matrix $D^2u$ converge to some constant matrix $A\in\mathtt{Sym}(n)$ at infinity, in order to make preparation for proving Theorem \ref{Thm-firstExpansion}. In the next two sections we give the proofs of  Theorems \ref{Thm-firstExpansion} and \ref{Thm-secondExpansion}   respectively based on the detailed analysis of the solutions of non-homogeneous linearized equations.

Hereinafter, we let $B_r(x)$ denote a ball centered at $x\in\mathbb R^n$ with radius $r$. Especially for $x=0$, we let $B_r:=B_r(0)$. For any open subset $\Omega\subset\mathbb R^n$, we let $\overline{\Omega}$ denote the closure of $\Omega$ and $\Omega^c$ denote the complement of $\Omega$ in $\mathbb R^n$.

\section{Convergence of Hessian at infinity}\label{sec-convergeHessian}

In this section, we study the asymptotic behavior at infinity of Hessian matrix of classical solutions of \eqref{Equ-exterior}. We prove a weaker convergence than \eqref{equ-asym-Behavior} in Theorem \ref{Thm-firstExpansion} and $D^2u$ has bounded $C^{\alpha}$ norm for some $0<\alpha<1$ under a weaker assumption on $f$.
By
interior regularity as Lemma 17.16 of \cite{GT} and
 extension theorem as Theorem 6.10 of \cite{EvansMeasureTheory}, we
  may change the value of $u, f$ on a compact subset of $\mathbb R^n$ and prove only  for $u\in C^{2,\alpha}(\mathbb R^n)$ and $f\in C^{\alpha}(\mathbb R^n)$.
\begin{theorem}\label{thm-sec2}
  Let $u$ be as in Theorem \ref{Thm-firstExpansion}, $f\in C^{\alpha}(\mathbb R^n)$ for some $0<\alpha<1$ and satisfy
    \begin{equation}\label{equ-temp-6}
    \limsup_{|x|\rightarrow\infty}\left(
    |x|^{\zeta}|f(x)-f(\infty)|+
    |x|^{\alpha+\zeta'}[f]_{C^{\alpha}
    (\overline{B_{\frac{|x|}{2}}(x)})}\right)<\infty
    \end{equation}
  \begin{enumerate}[(1)]
    \item
    with some $\zeta>1,\zeta'>0$ for $\tau=0$;
    \item \label{case-2.1-2}
    with some $\zeta>1,\zeta'>0$  for $\tau\in(0,\frac{\pi}{4})$;
    \item
     with some $\zeta>0,\zeta'>0$  for  $\tau=\frac{\pi}{4}$.
  \end{enumerate}
 Then there exist $\epsilon>0, A\in\mathtt{Sym}(n)$ with $F_{\tau}(\lambda(A))=f(\infty)$ and $C>0$ such that
  $$
  ||D^2u||_{C^{\alpha}(\mathbb R^n)}\leq C,
  \quad\text{and}\quad \left|D^2u(x)-A\right|\leq \dfrac{C}{|x|^{\epsilon}},\quad\forall~|x|\geq 1.
  $$

\end{theorem}

The proof is separated into three subsections according to three different range of $\tau$.

\subsection{$\tau=0$ case}

In $\tau=0$ case, \eqref{Equ-exterior} becomes the Monge-Amp\`ere equation \eqref{equ-MA}.

\begin{theorem}\label{roughestimate}
  Let $u \in C ^ { 0 } \left( \mathbb { R } ^ { n } \right)$  be a convex viscosity solution of
  \begin{equation}\label{Monge-Ampere}
\operatorname{det} D^{2} u=\psi(x)\quad\text{in } \mathbb{R}^{n}
\end{equation}
  with $u(0)=\min_{\mathbb{R}^n}u=0$, where $
  0<\psi\in C ^ { 0 } \left( \mathbb { R } ^ { n } \right)
  $
  and $$\psi^{\frac{1}{n}}-1\in L^n(\mathbb{R}^n).
$$
  Then there exists a linear transform $T$ satisfying $\det T=1$ such that $v:=u\circ T$ satisfies
  $$
  \left|v-\dfrac{1}{2}|x|^2\right|\leq C|x|^{2-\varepsilon},\quad \forall~ |x|\geq 1.
  $$
  for some $C>0$ and $\varepsilon>0$.
\end{theorem}
Theorem \ref{roughestimate} can be found in the proof of Theorem 1.2 in \cite{BLZ}, which is based on the level set method by Caffarelli-Li \cite{CL}.

\begin{corollary}\label{corollary_estimate}
Let $u\in C^0(\mathbb R^n)$ be a convex viscosity solution of \eqref{Equ-exterior} with $f\in C^0(\mathbb R^n)$ satisfies
  $$
  \limsup_{|x|\rightarrow\infty}|x|^{\zeta}|f(x)-f(\infty)|<\infty
  $$
  for some  $\zeta>1$. Then there exists a linear transform $T$ satisfying $\det T=1$ such that $v:=u\circ T$ satisfies
   \begin{equation}\label{roughestimate_appendix}
  \left|v-\dfrac{\exp(f(\infty))}{2}|x|^2\right|\leq C|x|^{2-\varepsilon},\quad \forall~ |x|\geq 1
  \end{equation}
  for some $C>0$ and $\varepsilon>0$.
\end{corollary}

\begin{proof}
By a direct computation,
$$
\widetilde u(x):=\dfrac{1}{\exp(f(\infty))}\left(u(x)-Du(0)x-u(0)\right)
$$
is a convex viscosity solution of
$$
\det D^2\widetilde u=e^{n(f(x)-f(\infty))}=:\widetilde f(x)\quad\text{in }\mathbb R^n.
$$
By a direct computation, $|\widetilde  f(x)-1|\leq C|x|^{-\zeta}$ for some $C>0$ and
$$
  \int _ { \mathbb{R}^n\setminus B_{1} } \left|(\widetilde f(x))^ { \frac { 1 } { n } } - 1 \right| ^ { n } d x  \leq C
 \int _ { \mathbb{R}^n\setminus B_{1} } \left| \widetilde f (x ) - 1 \right| ^ { n } d x \leq C \int _ { \mathbb { R } ^ { n } \backslash B _ { 1 } } |x|^{-\zeta n} d x <\infty.
$$
The result follows immediately by applying Theorem \ref{roughestimate} to $\widetilde u$.
 \end{proof}

 As a consequence, we have the following convergence of Hessian matrix for solutions of \eqref{Monge-Ampere}. The proof is similar to the one in Bao-Li-Zhang \cite{BLZ} and in Caffarelli-Li \cite{CL}. Since there are some differences from their proof, we provide the details here for reading simplicity.

\begin{theorem}\label{corollary_estimate3}
  Let $u \in C ^ { 0 } \left( \mathbb { R } ^ { n } \right)$ be a convex viscosity solution of \eqref{Equ-exterior},  $
  f \in C^{\alpha}(\mathbb{R}^n ) $ satisfy \eqref{equ-temp-6} for some $0<\alpha<1$, $\zeta>1$ and $\zeta'>0$.
  Then $u\in C^{2,\alpha}(\mathbb R^n)$,
  \begin{equation}\label{HolderRegularity_MA}
    ||D^2u||_{C^{\alpha}(\mathbb R^n)}\leq C,
  \end{equation}
  and
  \begin{equation}\label{ConvergeRateofHessian}
  u-\left(\frac{1}{2}x^TAx+b x+c\right)=O_2(|x|^{2-\epsilon})
  \end{equation}
as $|x|\rightarrow \infty$,
where
$\epsilon:=\min\{\varepsilon,\zeta,\zeta'\}$,
 $\varepsilon$ is the positive constant from Theorem \ref{roughestimate}, $A\in\mathtt{Sym}(n)$ with $\det A=\exp(nf(\infty))$, $b\in\mathbb R^n$, $c\in\mathbb R$ and $C>0$.
\end{theorem}

\begin{proof}

  By  Corollary \ref{corollary_estimate}, there exist a linear transform $T$, $\varepsilon>0$ and $C>0$  such that $v:=u\circ T$ satisfies (\ref{roughestimate_appendix}).

  \textbf{Step 1:} prove $C^{\alpha}$ boundedness of Hessian \eqref{HolderRegularity_MA}.
Let
\begin{equation*}
v_{R}(y)=\left(\frac{4}{R}\right)^{2} v\left(x+\frac{R}{4} y\right), \quad|y| \leq 2
\end{equation*}
for $|x|=R>2$.
By (\ref{roughestimate_appendix}),
\begin{equation*}
\left\|v_{R}\right\|_{C^0\left(\overline{B_{2}}\right)} \leq C
\end{equation*}
for some $C>0$ for all $R\geq 2$.  Then $v_R$ satisfies
\begin{equation}\label{equ-temp-4}
\operatorname{det}\left(D^{2} v_{R}(y)\right)=\exp\left(nf\left(x+\frac{R}{4} y\right)\right)=: f_{R}(y) \quad \text { in } B_{2}.
\end{equation}
By a direct computation, there exists $C>0$ uniform to $x$ such that
$$
||f_{R}-\exp(nf(\infty))||_{C^{0}(\overline{B_2})}\leq C R^{-\zeta}
$$
and for all $y_1,y_2\in B_2$,
$$
\dfrac{|f_R(y_1)-f_R(y_2)|}{|y_1-y_2|^{\alpha}}=
\dfrac{|f(z_1)-f(z_2)|}{|z_1-z_2|^{\alpha}}\cdot (\frac{R}{4})^{\alpha}\leq CR^{-\zeta'},
$$
where $z_i:=x+\frac{R}{4}y_i\in B_{\frac{|x|}{2}}(x)$.
Applying the interior estimate by Caffarelli \cite{7}, Jian-Wang \cite{25} on $B_2$, we have
\begin{equation}\label{CalphaEstimate}
\left\|D^{2} v_{R}\right\|_{C^{\alpha}\left(\overline{B_{1}}\right)} \leq C
\end{equation}
and hence
\begin{equation}\label{2.9}
\frac{1}{C}I\leq D^{2} v_{R} \leq C I\quad \text {in}~ B_{1}
\end{equation}
for some $C$ independent of $R$.
For any $|x|=R\geq 2$, we have
\begin{equation}\label{equ-bddHessian}
|D^2v(x)|=|D^2v_R(0)|\leq
||D^2v_R||_{C^0(\overline{B_1})}\leq C.
\end{equation}
For any $x_1,x_2\in B_{2}^c$ with $0<|x_2-x_1|\leq \frac{1}{4}|x_1|$, let $R:=|x_1|>2$, by (\ref{CalphaEstimate}),
\begin{equation*}
\begin{array}{llll}
\dfrac{\left|D^{2} v\left(x_{1}\right)-D^{2} v\left(x_{2}\right)\right|}{\left|x_{1}-x_{2}\right|^{\alpha}}
&=&\dfrac{\left|D^{2} v_{R}\left(0\right)-D^{2} v_{R}\left(\frac{4(x_2-x_1)}{|x_1|}\right)\right|}{\left|x_{1}-x_{2}\right|^{\alpha}}\\
&\leq & [D^2v_{R}]_{C^{\alpha}(\overline{B_1})}\cdot \left(\frac{4}{|x_1|}\right)^{\alpha}\\
&\leq & CR^{-\alpha}.\\
\end{array}
\end{equation*}
For any $x_1,x_2\in B_{2}^c$ with $|x_2-x_1|\geq \frac{1}{4}|x_1|$, by \eqref{equ-bddHessian},
$$
\dfrac{|D^2v(x_1)-D^2v(x_2)|}{|x_1-x_2|^{\alpha}}
\leq 2^{\alpha}\cdot 2||D^2v||_{C^0(\mathbb R^n)}\leq C.
$$
Since the linear transform $T$ from Theorem \ref{roughestimate} is invertible, (\ref{HolderRegularity_MA}) follows immediately.

\textbf{Step 2:} prove convergence speed at infinity \eqref{ConvergeRateofHessian}.
Let
$$
w(x):=v(x)-\dfrac{\exp(f(\infty))}{2}|x|^2\quad\text{and}\quad w_R(y):=
\left(\frac{4}{R}\right)^{2} w\left(x+\frac{R}{4} y\right), \quad|y| \leq 2
$$
for $|x|=R\geq 2$. By (\ref{roughestimate_appendix}) in Theorem \ref{roughestimate},
\begin{equation*}
\left\|w_{R}\right\|_{C^0\left(\overline{B_{2}}\right)} \leq C R^{-\varepsilon}.
\end{equation*}
Applying Newton-Leibnitz formula between \eqref{equ-temp-4} and $\operatorname{det}(\exp(f(\infty))I)=\exp(nf(\infty))$,
\begin{equation*}
\widetilde{a_{i j}}(y) D_{i j} w_{R}=f_{ R}(y)-\exp(nf(\infty))\quad \text{in }B_2,
\end{equation*}
where $\widetilde{a_{i j}}(y)=\int_{0}^{1} D_{M_{i j}}(\det \left(I+t D^{2} w_{R}(y)\right)) d t$.

By (\ref{CalphaEstimate}) and (\ref{2.9}), there exists constant $C$ independent of $|x|=R>2$ such that
\begin{equation*}
\frac{I}{C} \leq \widetilde{a_{i j}} \leq C I\quad\text {in } B_{1}, \quad\left\|\widetilde{a_{i j}}\right\|_{C^{ \alpha}\left(\overline{B_{1}}\right)} \leq C.
\end{equation*}
By interior Schauder estimates, see for instance Theorem 6.2 of \cite{GT},
\begin{equation}\label{equ-interiorSchauder}
\begin{array}{llll}
\left\|w_{R}\right\|_{C^{2, \alpha}\left(\overline{B_{\frac{1}{2}}}\right)} &\leq & C\left(\left\|w_{R}\right\|_{C^0\left(\overline{B_{1}}\right)}+\left\|f_{ R}-\exp(nf(\infty))\right\|_{C^{\alpha}\left(\overline{B_{1}}\right)}\right)\\
& \leq & C R^{-\min\{\varepsilon,\zeta,\zeta'\}}.\\
\end{array}
\end{equation}
The result (\ref{ConvergeRateofHessian}) follows immediately by
scaling back.
 \end{proof}

\begin{remark}
  In the proof of Theorem \ref{corollary_estimate3}, the interior Schauder estimates used in \eqref{equ-interiorSchauder} can be replaced by the $W^{2,\infty}$ type estimates (see for instance Remark 1.3 of \cite{Dong-Xu-estimate}),
  $$
  ||w_R||_{W^{2,\infty}(\overline{B_{\frac{1}{2}}})}\leq C \left(
  \left\|w_{R}\right\|_{C^0\left(\overline{B_{1}}\right)}+\left\|f_{ R}-\exp(nf(\infty))\right\|_{C^{\alpha}\left(\overline{B_{1}}\right)}
  \right)\leq C R^{-\min\{\varepsilon,\zeta,\zeta'\}}.
  $$
\end{remark}

\begin{remark}\label{corollary_estimate2}

The condition \eqref{equ-temp-6} in Theorem \ref{corollary_estimate3} holds if
for some $C>0$,
 \begin{equation}\label{1-orderCondition}
|x|^{\zeta}|f(x)-f(\infty)|+|x|^{1+\zeta'}
  \left| Df ( x ) \right|\leq C,\quad\forall~|x|>2.
  \end{equation}

Even if $f(x)$ is $C^1$, condition \eqref{equ-temp-6} is weaker than (\ref{1-orderCondition}).
For example,  we consider $f(x):=e^{-|x|}\sin(e^{|x|})$. On the one hand,  $Df(x)$ doesn't admit a limit at infinity, hence $f$ doesn't satisfy condition (\ref{1-orderCondition}).
On the other hand,
for any
$|x|=R>1$ and
$z_1,z_2\in B_{\frac{|x|}{2}}(x)$,
$$
\begin{array}{lll}
\dfrac{\left|f\left(z_{1}\right)-f\left(z_{2}\right)\right|}{\left|z_{1}-z_{2}\right|^{\alpha}}
&\leq & e^{-|z_2|}\dfrac{\left|\sin(e^{|z_1|})-\sin(e^{|z_2|})\right|}{\left|z_{1}-z_{2}\right|^{\alpha}}
+ \sin(e^{|z_1|}) \dfrac{\left|e^{-|z_1|}-e^{-|z_2|}\right|}{\left|z_{1}-z_{2}\right|^{\alpha}}\\
&\leq &\displaystyle C e^{-\frac{R}{2}}\cdot \frac{\left|z_{1}-z_{2}\right|}{\left|z_{1}-z_{2}\right|^{\alpha}} \\
&\leq & Ce^{-\frac{R}{2}}\cdot R^{1-\alpha}\\
\end{array}
$$
for constant $C$ independent of $R$. Hence $f$ satisfies condition \eqref{equ-temp-6} for all $\alpha\in(0,1)$ and any $\zeta,\zeta'>0$.
\end{remark}

This finishes the proof of Theorem \ref{thm-sec2} for $\tau=0$ case.

\subsection{$\tau\in(0,\frac{\pi}{4})$ case}

In this subsection, we deal with $\tau\in(0,\frac{\pi}{4})$ case by Legendre transform and the results in previous subsection.

Let
$f\in C^{\alpha}(\mathbb R^n)$  satisfy \eqref{equ-temp-6}  for some $0<\alpha<1, \zeta>1$, $\zeta'>0$ and $u\in C^{2,\alpha}(\mathbb R^n)$ be a classical solution of \eqref{Equ-exterior} satisfying \eqref{case-small}.
Let
$$
\overline{u}(x):=u(x)+\dfrac{a+b}{2}|x|^2,
$$
then
\begin{equation}\label{equ-star-1}
D^{2} \overline{u}=D^{2} u+(a+b) I>2bI\quad\text{in }\mathbb{R}^n.
\end{equation}
Let $(\widetilde{x},v)$ be the Legendre transform of $(x,\overline{u})$, i.e.,
\begin{equation}\label{LegendreTransform2}
\left\{
\begin{array}{ccc}
  \widetilde{x}:=D\overline{u}(x),\\
  Dv(\widetilde{x}):=x,\\
\end{array}
\right.
\end{equation}
and  we have
\begin{equation*}
D ^{2} v(\widetilde{x})=\left(D^{2} \overline{u}(x)\right)^{-1}=(D^2u(x)+(a+b)I)^{-1}<\frac{1}{2b}I.
\end{equation*}
Let \begin{equation}\label{LegendreTransform}
\bar v(\widetilde{x}):=\dfrac{1}{2}|\widetilde{x}|^2-2bv(\widetilde{x}).
\end{equation}
By a direct computation, $D\bar u(\mathbb R^n)=\mathbb R^n$ and
\begin{equation}\label{property-Legendre}
\widetilde{\lambda}_{i}\left(D^{2} \bar v\right)=1-2 b \cdot \frac{1}{\lambda_{i}+a+b}=\frac{\lambda_{i}+a-b}{\lambda_{i}+a+b}\in (0,1).
\end{equation}
Thus $\bar v(\widetilde{x})$ satisfies the following Monge-Amp\`ere type equation
\begin{equation}\label{temp-16}
\operatorname{det} D^{2} \bar v=\exp \left\{\frac{2 b}{\sqrt{a^{2}+1}} f\left(\frac{1}{2 b}(\widetilde{x}-D \bar v(\widetilde{x}))\right)\right\}=: g(\widetilde{x})\quad \text{in }\mathbb{R}^n.
\end{equation}

\textbf{Step 1:} There exists $C_0>1$ such that
\begin{equation}\label{linear-of-X}
\frac{1}{C_0}|x|\leq |\widetilde{x}|\leq C_0|x|,\quad\forall~|x|> 1.
\end{equation}
We prove the two inequalities in \eqref{linear-of-X} separately.

By the definition of $\widetilde x=D\bar u(x)$ and \eqref{equ-star-1},
\begin{equation*}
|\widetilde{x}-\widetilde{0}|=|D \bar u(x)-D \bar u(0)|>2 b|x|.
\end{equation*}
Hence by triangle inequality,
\begin{equation}\label{limitofX}
|\widetilde{x}|\geq
-|\widetilde{0}|+|\widetilde{x}-\widetilde{0}|
> -|\widetilde{0}|+2b|x|,
\end{equation}
and the first inequality of \eqref{linear-of-X} follows immediately.

 By the quadratic growth condition in \eqref{Condition-QuadraticGrowth-1}, we prove the linear growth result of $Du(x)$.  In fact, for any $|x|\geq 1$, let
 $e:=\frac{Du(x)}{|Du(x)|}\in\partial B_1$. By Newton-Leibnitz formula and \eqref{Condition-QuadraticGrowth-1}
 \begin{equation}\label{gradientestimate}
 \begin{array}{lllll}
 u(x+|x|e)&= &\displaystyle u(x)+ \int_0^{|x|}e\cdot Du(x+se)\mathtt{d}s\\
 &=& \displaystyle u(x)+\int_0^{|x|}\int_0^se\cdot D^2u(x+te)\cdot e\mathtt{d}t\mathtt{d} s +\int_0^{|x|}e\cdot Du(x)\mathtt{d}s\\
 &\geq & \displaystyle u(x)+\frac{(-a+b)}{2}|x|^2+|Du(x)|\cdot |x|.\\
 \end{array}
 \end{equation}
 Furthermore by \eqref{Condition-QuadraticGrowth-1}, there exists $C>0$ independent of $|x|\geq 1$ such that
 $$
 |Du(x)|\leq \dfrac{1}{|x|} \left(
 C(1+\left|(x+|x|e)\right|^2)+C(1+|x|^2)+\frac{a-b}{2}|x|^2
  \right)
 \leq C(1+|x|).
 $$
 Hence there exists $C>0$ such that
\begin{equation}\label{Condition-LinearGrowth}
|Du(x)|\leq C(1+|x|),\quad \forall~x\in\mathbb{R}^n.
\end{equation}
By \eqref{Condition-LinearGrowth}, there exists $C>0$ such that
$$
|\widetilde x|=|D u(x)+(a+b) x| \leq|D u(x)|+(a+b)|x| \leq C(|x|+1).
$$
The second inequality of \eqref{linear-of-X} follows immediately.

Now we study   equation \eqref{temp-16} by applying Theorem \ref{corollary_estimate3} and Remark \ref{corollary_estimate2}, which require a knowledge on the asymptotic behavior of $g(\widetilde x)$.

\textbf{Step 2: }$g(\widetilde{x})$ satisfies condition \eqref{equ-temp-6}.
By the equivalence (\ref{linear-of-X}), $$
  \lim_{\widetilde{x}\rightarrow\infty}g(\widetilde{x})=\exp\left\{\frac{2 b}{\sqrt{a^{2}+1}}f(\infty)\right\}=:g(\infty)\in(0,1].
  $$
  By a direct computation,
   $$
   \begin{array}{lllll}
  &\displaystyle
  |\widetilde{x}|^{\zeta}|g(\widetilde{x})-g(\infty)|\\
  =& \displaystyle e^{\frac{2 b}{\sqrt{a^{2}+1}}f(\infty)}
  \dfrac{ |\widetilde{x}|^{\zeta}}{\left|\frac{\widetilde{x}-D\bar v(\widetilde{x})}{2b}\right|^{\zeta}}
  \cdot\left|\frac{\widetilde{x}-D\bar v(\widetilde{x})}{2b}\right|^{\zeta}
  \cdot \left|
  e^{\frac{2b}{\sqrt{a^2+1}}(f(\frac{\widetilde{x}-D\bar v(\widetilde{x})}{2b})-f(\infty))}-1
  \right|
  \\
  \leq &
  \displaystyle C |x|^{\zeta}\left|
  e^{\frac{2b}{\sqrt{a^2+1}}(f(x)-f(\infty))}-1
  \right|
  \\
  \leq &\displaystyle C
  |x|^{\zeta}\left|f(x)-f(\infty)\right|
  <C.\\
  \end{array}
  $$
  For any $\widetilde  y,\widetilde  z\in B_{\frac{|\widetilde  x|}{2}\cdot 2b}(\widetilde  x), \widetilde  y\not=\widetilde  z$ with   $|\widetilde  x|> C_0$, by \eqref{equ-star-1} we have
  $$
  y,z\in B_{\frac{|x|}{2}}(x),\quad
  |\widetilde  y-\widetilde  z|\geq 2b|y-z|>0\quad\text{and}\quad y\not=z.
  $$
  Thus by condition \eqref{equ-temp-6},
  \begin{equation}\label{equ-temp-10}
  \dfrac{
  |g(\widetilde  y)-g(\widetilde  z)|
  }{|\widetilde  y-\widetilde  z|^{\alpha}}\leq (2b)^{-\alpha}\dfrac{\exp\{\frac{2b}{\sqrt {a^2+1}}f(y)\}
  -\exp\{\frac{2b}{\sqrt {a^2+1}}f(z)\}
  }{|y-z|^{\alpha}}\leq C[f]_{C^{\alpha}(\overline{B_{\frac{|x|}{2}}(x)})}.
  \end{equation}
  Thus $g(\widetilde x)$ satisfies
  \eqref{equ-temp-6} for $0<\alpha<1, \zeta>1$ and $\zeta'>0$ as given.

By Theorem \ref{corollary_estimate3}, we have
$$
||D^2\bar v||_{C^{\alpha}(\mathbb R^n)}\leq C
$$
and
\begin{equation}\label{equ-temp-1}
\bar v-\left(\frac{1}{2}\widetilde  x^T\widetilde  A\widetilde  x+\widetilde  b\cdot \widetilde  x+\widetilde  c\right)=O_2(|\widetilde x|^{2-\epsilon})
\end{equation}
for some $0<\widetilde{A}  \in \mathtt{Sym}(n)$ satisfying $\det \widetilde{A}=g(\infty)$, $\widetilde b\in\mathbb R^n, \widetilde  c\in \mathbb R$ and $C,\epsilon > 0$.

\textbf{Step 3: } we finish the proof of Theorem \ref{thm-sec2} \eqref{case-2.1-2}.
By strip argument as in \cite{ExteriorLiouville,bao-liu-2020} etc, we prove that $I-\widetilde{A}$ is invertible. In fact,
by \eqref{property-Legendre}, $\widetilde A\leq I$ and it remains to prove $\lambda_i(\widetilde A)<1$ for all $i=1,2,\cdots,n$.
Arguing by contradiction and rotating the $\widetilde{x}$-space to make $\widetilde{A}$ diagonal, we may assume that $\widetilde{A}_{11}=1$. By \eqref{equ-temp-1} with the definition of Legendre transform (\ref{LegendreTransform}) and (\ref{limitofX}), there exists $\widetilde {b_1}$ such that
  \begin{equation}\label{strip-argument}
  x_1=D_1v(\widetilde  x)=\widetilde {b_1}+O(|\widetilde  x|^{1-\epsilon})
  \quad\text{as }|\widetilde  x|\rightarrow\infty.
\end{equation}
This becomes a contradiction to \eqref{linear-of-X}.

Let $$
A:=2b\left( I - \widetilde{A} \right)   ^ { - 1 } - ( a + b ) I.
$$
By a direct computation, $F_{\tau}(\lambda(A))=f(\infty)$ and $$
 \begin{array}{llll}
\left| D ^ { 2 } u(x) - A \right|& =&2b\left|
\left( I - D ^ { 2 } \bar v ( \widetilde { x } )\right) ^ { - 1 }-
\left( I - \widetilde{A} \right) ^ { - 1 }
\right|\\
&\leq &C|D^2\bar v(\widetilde x)-\widetilde{A}|\\
&
\leq& \dfrac{C}{|\widetilde{x}|^{\epsilon}}\quad\forall~|x|\geq 1.\\
\end{array}
$$
By the equivalence (\ref{linear-of-X}),
we have
  \begin{equation}\label{Result_LimitofHessian}
  \left| D ^ { 2 } u ( x ) - A \right| \leq \frac { C } { | x | ^ {\epsilon} } ,\quad\forall ~| x | \geq 1.
  \end{equation}
Furthermore,
  by (\ref{LegendreTransform}), for any $x,y\in\mathbb{R}^n$,
\begin{equation*}
\left|D^{2} u(x)-D^{2} u(y)\right|=2 b\left|\left(I-D^{2} \bar v(\widetilde{x})\right)^{-1}-\left(I-D^{2} \bar v(\widetilde{y})\right)^{-1}\right|.
\end{equation*}
By \eqref{Result_LimitofHessian}, $D^2\bar v(\widetilde x)$ is bounded away from $0$ and $I$, it follows that $\exists~ C>0$ such that
\begin{equation}\label{equivalentHessian}
\left|D^{2} u(x)-D^{2} u(y)\right| \leq 2 b C\left|D^{2} \bar v(\widetilde{x})-D^{2} \bar v(\widetilde{y})\right|
\end{equation}
Combining (\ref{equivalentHessian}) and the equivalence (\ref{linear-of-X}),
$D^2u$ has bounded $C^{\alpha}$ norm.

So far, we
 finished the proof of Theorem \ref{thm-sec2} for $\tau\in (0,\frac{\pi}{4})$ case.

\subsection{$\tau=\frac{\pi}{4}$ case}

In this subsection, we deal with $\tau=\frac{\pi}{4}$ case by Legendre transform and analysis on the  Poisson equations.

Let $f\in C^{\alpha}(\mathbb R^n)$ satisfy \eqref{equ-temp-6} for some $0<\alpha<1,\zeta,\zeta'>0$  and
$u\in C^{2,\alpha}(\mathbb R^n)$ be a classical solution  of \eqref{Equ-exterior} satisfying \eqref{case-inverse}.
Let $$
\overline{u}(x):=u(x)+\dfrac{1}{2}|x|^2,
$$
then $D^2\overline u>0$ in $\mathbb R^n$.
By equation \eqref{Equ-exterior}, for all $i=1,2,\cdots,n$,
$$
-\dfrac{1}{\lambda_i(D^2\bar u)}\geq -\sum_{j=1}^n\dfrac{1}{\lambda_j(D^2\bar u)}\geq \frac{\sqrt 2}{2}\inf_{\mathbb R^n}f.
$$
Thus
 there exists $\delta>0$ such that
$$
D^2\overline u(x)>\delta I,\quad\forall~x\in\mathbb R^n.
$$
Let $(\widetilde{x},v)$ be the Legendre transform of $(x,\overline{u})$ as in \eqref{LegendreTransform2}
and we have
\begin{equation*}
0<D ^2v(\widetilde{x})=(D^2\overline{u}(x))^{-1}<\dfrac{1}{\delta}I.
\end{equation*}
By a direct computation, $D\bar u(\mathbb R^n)=\mathbb R^n$ and $v(\widetilde{x})$ satisfies the following Poisson equation \begin{equation}\label{temp-modify3}
\Delta v=-\frac{\sqrt{2}}{2}f(Dv(\widetilde  x))=:g(\widetilde{x})\quad\text{in }\mathbb R^n.
\end{equation}
\textbf{Step 1:} There exists $C_0>1$ such that \eqref{linear-of-X} holds. The proof is separated into two parts similarly.

By the definition of Legendre transform in \eqref{LegendreTransform2},
$$
|\widetilde x-\widetilde 0|=|D\bar u(x)-D\bar u(0)|>\delta |x|.
$$
Hence by triangle inequality,
$$
|\widetilde x|\geq -|\widetilde 0|+|\widetilde x-\widetilde 0|>-|\widetilde 0|+\delta|x|
$$
and the first inequality of \eqref{linear-of-X} follows immediately.
The second inequality of \eqref{linear-of-X} follows similarly by \eqref{Condition-QuadraticGrowth-2} and
 \eqref{gradientestimate}.

 \textbf{Step 2:} Asymptotic behavior of $g(\widetilde x)$ at infinity. By  the equivalence \eqref{linear-of-X},
$$
g(\widetilde  x)=-\frac{\sqrt 2}{2}f(x)\rightarrow -\frac{\sqrt 2}{2}f(\infty)=:g(\infty)
$$
as $|\widetilde  x|\rightarrow+\infty$. Similar to the proof of \eqref{equ-temp-10},
we have
$$
\limsup_{|\widetilde x|\rightarrow+\infty}\left(|\widetilde x|^{\zeta}|g(\widetilde x)-g(\infty)|+|\widetilde x|^{\alpha+\zeta'}[g]_{C^{\alpha}(\overline{B_{\frac{|\widetilde  x|}{2}}(\widetilde  x)})}\right)<\infty
$$
for the give $0<\alpha<1$, $\zeta,\zeta'>0$.

 \textbf{Step 3:} Asymptotic behavior of $v(\widetilde x)$ at infinity.

Since \eqref{equ-temp-6} remains when $\zeta>0$ becomes smaller, we only need to prove for $0<\zeta<2$ case for reading simplicity.
By a direct computation,  $\Delta |x|^{2-\zeta}=c_{n,\zeta}|x|^{-\zeta}$ in $B_1^c$.
Thus there exist subsolution $\underline v$ and supersolution $\overline v$ of
Poisson equation
\begin{equation}\label{equ-Poisson}
\Delta \widetilde  v=g(\widetilde x)-g(\infty)\quad\text{in }\mathbb R^n
\end{equation}
with
$
\underline{v},\overline v=O(|\widetilde x|^{2-\zeta})$ as $|x|\rightarrow\infty.
$
By Perron's method (see for instance \cite{BLZ,UserGuide-intro,ExteriorDirichlet}) and interior regularity, we have a classical solution $\widetilde v\in C^{2,\alpha}(\mathbb R^n)$ of \eqref{equ-Poisson} with $\widetilde v=O(|\widetilde x|^{2-\zeta})$ as $|\widetilde x|\rightarrow\infty$.

For any $|\widetilde  x|=R\geq 1$, let
$$
 \widetilde  v_R(y):=\left(\frac{2}{R}\right)^2  \widetilde v(\widetilde  x+\frac{R}{2}y),\quad y\in B_1.
$$
Then $ \widetilde  v_R$ satisfies
$$
\Delta \widetilde  v_R=g(\widetilde  x+\frac{R}{2}y)-g(\infty)=:g_R(y)\quad\text{in }B_1.
$$
By a direct computation,
$$
||g_R||_{C^{\alpha}(\overline{B_1})}\leq CR^{-\min\{\zeta,\zeta'\}}
\quad\text{and}\quad
|| \widetilde  v_R||_{C^0(\overline{B_1})}\leq CR^{-\zeta}.
$$
By interior Schauder estimates, we have
$$
||\widetilde v_R||_{C^{2,\alpha}(\overline{B_{1/2}})}\leq CR^{-\min\{\zeta,\zeta'\}}
$$
and then
$$
\widetilde v(\widetilde x)=O_2(|\widetilde x|^{2-\min\{\zeta,\zeta'\}})
$$
as $|\widetilde x|\rightarrow\infty$.
Then
$$
\Delta (v-\widetilde v)=g(\infty)\quad\text{in }\mathbb R^n
$$
and $D^2(v-\widetilde  v)$ is bounded.  By Liouville type theorem, $v-\widetilde v$ is a quadratic function and hence
$$
v-\left(
\frac{1}{2}\widetilde x^T\widetilde A\widetilde x+\widetilde b \widetilde x+\widetilde c
\right)=O_2(|\widetilde x|^{2-\min\{\zeta,\zeta'\}})
$$
for some  $\widetilde  A\in\mathtt{Sym}(n)$ with $\mathtt{trace}\widetilde  A=g(\infty)$,  $\widetilde b\in\mathbb R^n$ and $\widetilde c\in\mathbb R$.
Similarly we have  \eqref{strip-argument} and $\widetilde A$ is invertible. Taking $A:=\widetilde A^{-1}-I$ and the result follows similar to  $\tau\in(0,\frac{\pi}{4})$ case.

\section{Asymptotics  of solutions of \eqref{Equ-exterior}}\label{sec-linear}

In this section, we
prove Theorem \ref{Thm-firstExpansion}. As an integral part of the preparation, we
 analyze the linearized equation of \eqref{Equ-exterior} and obtain the asymptotic behavior at infinity. The major difficulty is that the linearized equation is not homogeneous.

\subsection{Asymptotics of solutions of nonhomogeneous linear elliptic equations}

Consider the linear elliptic equation
  \begin{equation}\label{Dirichlet}
  Lu:=a_{i j}(x) D_{i j} u(x)=f(x)\quad\text{in }\mathbb R^n,
  \end{equation}
 where the coefficients are  uniformly elliptic,
  satisfying \begin{equation}\label{HolderCoefficient}
  ||a_{ij}||_{C^{\alpha}(\mathbb{R}^n)}<\infty,
  \end{equation}
  for some $0<\alpha <1$
  and
  \begin{equation}\label{short-RangeCoefficient}
  |a_{ij}(x)-a_{ij}(\infty)|\leq C|x|^{-\varepsilon},
  \end{equation}
  for some $0<(a_{ij}(\infty))\in\mathtt{Sym}(n)$ and $\varepsilon,C>0$.

\begin{theorem}\label{exteriorLiouville}
  Let $v$ be a  classical solution of \eqref{Dirichlet} that bounded from at least one side, the coefficients satisfy \eqref{HolderCoefficient} and (\ref{short-RangeCoefficient}) and $f\in C^{0}(\mathbb{R}^n)$ satisfy
  \begin{equation}
  \label{decayoff_1}
\limsup_{|x|\rightarrow+\infty} |x|^{\zeta}|f(x)|<\infty
  \end{equation}
  for some $\zeta>2$.
  Then there exists a constant $v_{\infty}$ such that \begin{equation}\label{Result_ExteriorLiouville-2}
  v ( x ) = v _ { \infty } +
  \left\{
  \begin{array}{llll}
  O \left( |x|^{2-\min\{n,\zeta\}} \right), & \zeta\not=n,\\
  O \left( |x|^{2-n}(\ln|x|) \right), & \zeta=n,\\
  \end{array}
  \right.
  \end{equation}
  as $|x|\rightarrow\infty$.
\end{theorem}
The homogeneous version of Theorem \ref{exteriorLiouville} has been proved earlier,
 see for instance Gilbarg-Serrin \cite{Gilbarg-Serrin} and  Li-Li-Yuan \cite{ExteriorLiouville}. Hence we start with constructing a special solution of \eqref{Dirichlet} and translate the question into homogeneous case.

By  the  criterion in \cite{Equivalence}, the Green's function of operator $L$ is equivalent to the Green's function of Laplacian under conditions \eqref{HolderCoefficient} and \eqref{short-RangeCoefficient}. More precisely, let $G_L(x,y)$ be the Green's function centered at $y$ ,   there exists constant $C$ such that
\begin{equation}\label{Equivalence-Green}
\begin{array}{llll}
C^{-1}|x-y|^{2-n} \leq G_{L}(x, y) \leq C|x-y|^{2-n},& \forall~ x\not=y,\\
\left|D_{x_{i}} G_{L}(x, y)\right| \leq C|x-y|^{1-n}, \quad i=1, \cdots, n,& \forall~ x\not=y,\\
\left|D_{x_{i}}D_{x_{j}} G_{L}(x, y)\right| \leq C|x-y|^{-n}, \quad i, j=1, \cdots, n,& \forall~ x\not=y.
\end{array}
\end{equation}
By an elementary estimate as in Bao-Li-Zhang \cite{BLZ}, we construct a solution that vanishes at infinity. More rigorously, we introduce the following result.
\begin{lemma}\label{existence}
  There exists a bounded strong solution $u\in W^{2,p}_{loc}(\mathbb R^n)$ with $p>n$ of  (\ref{Dirichlet}) satisfying  $$
  u(x)=
  \left\{
\begin{array}{lllll}
O(|x|^{2-\min\{n,\zeta\}}), & \zeta\not=n,\\
O(|x|^{2-n}(\ln|x|)), & \zeta =n,\\
\end{array}
\right.
  $$
  as $|x|\rightarrow\infty$.
\end{lemma}
\begin{proof}
 By \eqref{Equivalence-Green} and  Calder\'on-Zygmund inequality,
 $$
w(x):=\int_{\mathbb{R}^n}G_L(x,y)f(y)\mathtt{d}y
$$ belongs to $W^{2,p}_{loc}(\mathbb{R}^n)$ for $p>n$ and is a strong solution of \eqref{Dirichlet} (see for instance \cite{Adams,Ziemer}). It remains to compute the vanishing speed at infinity.
Let $$
\begin{array} { l } { E _ { 1 } : = \left\{ y \in \mathbb { R } ^ { n } , \quad | y | \leq | x | / 2  \right\} ,} \\ { E _ { 2 } : = \left\{ y \in \mathbb { R } ^ { n }  , \quad | y - x | \leq | x | / 2  \right\}, } \\ { E _ { 3 }: =  \mathbb { R } ^ { n }  \backslash \left( E _ { 1 } \cup E _ { 2 } \right) .} \end{array}
$$
By a direct computation,
$$
 \int_{E_1}\dfrac{1}{|x-y|^{n-2}}f(y)\mathtt{d}y
\leq  C\int_{B_{\frac{|x|}{2}}}f(y)\mathtt{d}y\cdot |x|^{2-n}
\leq \left\{
\begin{array}{lllll}
C|x|^{2-\min\{n,\zeta \}}, & \zeta\not=n,\\
C|x|^{2-n}(\ln|x|),& \zeta =n.\\
\end{array}
\right.
$$
Similarly, we have  $\frac{|x|}{2}\leq |y|$ in $E_2$ and hence
\begin{equation*}
 \int_{E_2}\dfrac{1}{|x-y|^{n-2}}f(y)\mathtt{d}y
\leq  C
  \int_{|x-y|\leq \frac{|x|}{2}}\dfrac{1}{|x-y|^{n-2}}\mathtt{d}y
  \cdot \dfrac{1}{|x|^{\zeta}}\leq C |x|^{2-\zeta}.
\end{equation*}
Now we separate $E_3$ into two parts
$$
E_3^+:=\{y\in E_3:|x-y|\geq |y|\},\quad E_3^-:=E_3\setminus E_3^+.
$$
Then
$$
\int_{E_3^+}\dfrac{1}{|x-y|^{n-2}\cdot|y|^{\zeta}}\mathtt{d}y
\leq \int_{|y|\geq \frac{|x|}{2}}
\dfrac{1}{|y|^{n+\zeta-2}}\mathtt{d}y\leq C|x|^{2-\zeta}
$$
and
$$
\int_{E_3^-}\dfrac{1}{|x-y|^{n-2}\cdot|y|^{\zeta}}\mathtt{d}y
\leq \int_{|y-x|\geq \frac{|x|}{2}}
\dfrac{1}{|y-x|^{n+\zeta -2}}\mathtt{d}y\leq C|x|^{2-\zeta}.
$$
 Hence there exists $C>0$ such that
 $$
 |w(x)|\leq C
\left| \int_{E_1\cup E_2\cup E_3}\dfrac{1}{|x-y|^{n-2}}f(y)\mathtt{d}y\right|
 \leq
 \left\{
\begin{array}{lllll}
C|x|^{2-\min\{n,\zeta\}}, & \zeta \not=n,\\
C|x|^{2-n}(\ln|x|), & \zeta =n.\\
\end{array}
\right.
 $$
 \end{proof}

\begin{proof}[Proof of Theorem \ref{exteriorLiouville}]

We may assume without loss of generality that $v$ is bounded from below, otherwise consider $-v$ instead.
  Let $w(x)$ be the bounded strong solution of (\ref{Dirichlet}) from Lemma \ref{existence}, then
  $$
  \widetilde v:=v-w-\inf_{\mathbb R^n}(v-w)\geq 0
  $$
  is a strong solution of  \eqref{Dirichlet} with $f\equiv0$.
 By interior regularity, $\widetilde{v}$ is a positive classical solution. By Theorem 2.2 in \cite{ExteriorLiouville},
 $$
 \widetilde{v}( x ) = \widetilde{v} _ { \infty } + O \left( |x|^{2-n}\right)\quad \text {as } | x | \rightarrow \infty,
$$
for some constant $\widetilde{v}_{\infty}$.
Then the result follows immediately from Lemma \ref{existence}.
 \end{proof}

\begin{remark}\label{ExteriorLiouville_NonPositive-remark}
  If $v$ is a classical solution of \eqref{Dirichlet} with  $|Dv(x)|=O(|x|^{-1})$ as $|x|\rightarrow\infty$ and  $f\in C^{0}(\mathbb{R}^n)$ satisfy
  (\ref{decayoff_1}), then $v$ is bounded from at least one side.  The proof is similar to $f\equiv 0$ case, which can be found in  Corollary 2.1 of \cite{ExteriorLiouville}.
\end{remark}

\subsection{Proof of Theorem \ref{Thm-firstExpansion}}\label{sec-Pf-1.1}

Let $u \in C^{2}\left(\mathbb{R}^{n}\right)$ be a classical solution of \eqref{Equ-exterior}, where $f$ satisfies \eqref{Low-Regular-Condition} for some $\zeta>2,m\geq 2$ and either of cases \eqref{case-MA}-\eqref{case-inverse} holds.
By extension and interior estimates, we may assume that $u\in W^{4,p}_{loc}(\mathbb R^n)$ for some $p>n$.
By Theorem \ref{thm-sec2}, Hessian matrix $D^2u$ have finite $C^{\alpha}$ norm on $\mathbb R^n$ and converge to some $A\in\mathtt{Sym}(n)$ at a H\"older speed as in \eqref{Result_LimitofHessian}.

Let $v : = u ( x ) - \frac { 1 } { 2 } x ^ { T } A x$. Applying Newton-Leibnitz formula between $$
F_{\tau}\left(\lambda\left(D^{2} v+A\right)\right)=f(x)\quad\text{and}
\quad F_{\tau}(\lambda(A))=f(\infty),
$$
we have
\begin{equation}\label{linearized-equation-3}
 \overline{a_{i j}}(x) D_{i j} v:=\int_{0}^{1} D_{M_{i j}} F_{\tau}\left(\lambda(t D^{2} v+A)\right) \mathrm{d} t \cdot D_{i j} v
 =f(x)-f(\infty)=:
\overline{f}(x)
\end{equation}
 For any $e\in\partial B_1$, by the concavity of operator $F$, the partial derivatives $v_e:=D_ev$ and $v_{ee}:=D^2_ev$ are strong solutions of
\begin{equation}\label{linearized-equation-1}
  \widehat{a_{ij}}(x)D_{ij}v_e:=D_{M_{i j}} F_{\tau} \left(\lambda( D^{2} v+A)\right) D_{i j} v_{e} =f_{e}(x),
\end{equation}
and
\begin{equation}\label{linearized-equation-2}
\widehat{a_{ij}}(x)D_{i j} v_{e e} \geq f_{e e}(x).
\end{equation}
By Theorem \ref{thm-sec2}, there exist $\epsilon>0$ and $C>0$ such that $$
\left| \overline { a _ { i j } } ( x ) - D_{M_{ij}}F_{\tau} (\lambda( A) ) \right| +\left| \widehat { a _ { i j } } ( x ) - D_{M_{ij}}F_{\tau}(\lambda(A)) \right| \leq \frac { C } { | x | ^ { \epsilon} }.
$$
By  condition (\ref{Low-Regular-Condition}) and constructing barrier functions  for \eqref{linearized-equation-2}, there exists $C>0$ such that for all $x\in\mathbb R^n$,
$$
v _ { e e } ( x ) \leq
\left\{
\begin{array}{llll}
C | x | ^ {2-\min\{n,\zeta+2\}}, & \zeta\not=n-2,\\
C|x|^{2-n}(\ln|x|), & \zeta=n-2.\\
\end{array}
\right.
$$
By the arbitrariness of $e$,
\begin{equation*}
\lambda_{\max }\left(D^{2} v\right)(x) \leq \left\{
\begin{array}{llll}
C | x | ^ {2-\min\{n,\zeta+2\}}, & \zeta\not=n-2,\\
C|x|^{2-n}(\ln|x|), & \zeta=n-2.\\
\end{array}
\right.
\end{equation*}
By \eqref{Low-Regular-Condition} and the ellipticity of equation (\ref{linearized-equation-3}),
\begin{equation*}
\lambda_{\min }\left(D^{2} v\right)(x) \geq-C \lambda_{\max }\left(D^{2} v\right)-C|\overline{f}(x)| \geq
\left\{
\begin{array}{llll}
-C | x | ^ {2-\min\{n,\zeta+2\}}, & \zeta\not=n-2,\\
-C|x|^{2-n}(\ln|x|), & \zeta=n-2.\\
\end{array}
\right.
\end{equation*}
Hence $$
\left| D ^ { 2 } v ( x ) \right| \leq
\left\{
\begin{array}{llll}
C | x | ^ {2-\min\{n,\zeta+2\}}, & \zeta\not=n-2,\\
C|x|^{2-n}(\ln|x|), & \zeta=n-2.\\
\end{array}
\right.$$

By Theorem \ref{thm-sec2}, the coefficients $\overline{a_{ij}},\ \widehat{a_{ij}}$ has bounded $C^{\alpha}$ norm on exterior domain. Since $\zeta>2$, applying  Remark \ref{ExteriorLiouville_NonPositive-remark} to equation \eqref{linearized-equation-1}, for any $e\in\partial B_1$, $v_e(x)$ is bounded from one side and   there exists $b _ { e } \in \mathbb { R }$ such that
\begin{equation}\label{capture1-order}  v _ { e } ( x ) = b_ { e } +
\left\{
\begin{array}{llll}
O \left( | x | ^ { 2 - \min\{n,\zeta+1\}} \right), & \zeta\not=n-1,\\
O \left( | x | ^ { 2 - n}(\ln|x|) \right), & \zeta=n-1,\\
\end{array}
\right.
\quad\text{as } | x | \rightarrow \infty.
\end{equation}
Picking $e$ as $n$ unit coordinate vectors of $\mathbb R^n$, we found $b\in\mathbb R^n$ from \eqref{capture1-order} and let
$$\overline { v } ( x ): = v ( x ) - b x = u ( x ) - \left( \frac { 1 } { 2 } x ^ { T } A x + b x \right).$$
By (\ref{capture1-order}),$$
|D\overline{v}(x)|=|(\partial_{x_1}v-b_1,\cdots,\partial_{x_n}v-b_n)|=
\left\{
\begin{array}{llll}
O \left( | x | ^ { 2 - \min\{n,\zeta+1\}} \right), & \zeta\not=n-1,\\
O \left( | x | ^ { 2 - n}(\ln|x|) \right), & \zeta=n-1,\\
\end{array}
\right.
$$
as $|x|\rightarrow\infty$.
By \eqref{linearized-equation-3},
\begin{equation*}
  \overline{a_{ij}}(x)D_{ij}\overline{v}=\overline{a_{ij}}(x)
  D_{ij}v=\overline{f}(x).
\end{equation*}
By the arguments above again, there exists $c\in\mathbb R$ such that
\begin{equation*}
   \overline{v}(x)=c+
   \left\{
   \begin{array}{llll}
   O(|x|^{2-\min\{n,\zeta\}}), & \zeta\not=n,\\
   O\left(|x|^{2-n}(\ln |x|)\right), & \zeta=n,\\
   \end{array}
   \right.
   \quad\text{as }|x|\rightarrow\infty.
\end{equation*}
Notice that here we used $\zeta>2$ for $|D\bar v|=O(|x|^{-1})$ and $\overline f=O(|x|^{-\zeta})$ at infinity.
Let
$
  Q(x):=\frac{1}{2}x^TAx+bx+c.
$ Then
$$
  |u-Q|=|\overline{v}-c|=
  \left\{
  \begin{array}{llll}
    O(|x|^{2-\min\{n,\zeta\}}), & \zeta\not=n,\\
    O(|x|^{2-n}(\ln|x|)), & \zeta=n,\\
  \end{array}
  \right.
  \text{ as }|x|\rightarrow\infty.
$$

Finally, we give the estimates of derivatives of $u$.
For  $|x|\geq 1$, let
\begin{equation*}
E(y)=\left(\frac{2}{|x|}\right)^{2}(u-Q)\left(x+\frac{|x|}{2} y\right).
\end{equation*}
Then by Newton-Leibnitz formula,
\begin{equation*}
\underline{a^{i j}}(y) D_{ij}E(y)=F_{\tau}\left(\lambda(A+D^{2} E(y))\right)-F_{\tau}(\lambda(A))=f(x+\frac{|x|}{2}y)-f(\infty)=:\underline f(y)\quad\text{in }B_1,
\end{equation*}
where
\begin{equation*}
\underline{a^{i j}}(y)=\int_{0}^{1} D_{M_{i j}}F_{\tau}\left(\lambda(A+t D^{2} E(y))\right) \mathtt{d}t.
\end{equation*}
By the Evans-Krylov estimate and interior Schauder estimate (see for instance Chap.8 of \cite{FullyNonlinear} and Chap.6 of \cite{GT}), for all $0<\alpha<1$,
 we have
$$
\begin{array}{llll}
||E||_{C^{2,\alpha}(\overline{B_{\frac{1}{2}}})}&\leq & C(||E||_{C^0(\overline{B_1})}+||\underline f||_{C^{\alpha}(\overline{B_2})})\\
&\leq & C(||E||_{C^0(\overline{B_1})}+||\underline f||_{C^{1}(\overline{B_2})})\\
&=&
  \left\{
  \begin{array}{llll}
    O(|x|^{-\min\{n,\zeta\}}), & \zeta\not=n,\\
    O(|x|^{-n}(\ln|x|)), & \zeta=n,\\
  \end{array}
  \right.
  \text{ as }|x|\rightarrow\infty.
\end{array}
$$
By taking further derivatives and iterate, we have for all $k\leq m+1$,
$$
\begin{array}{lllll}
\left(\frac{|x|}{2}\right)^{k-2}\left|D^k(u-Q)(x)\right|&=&|D^kE(0)|\\
&\leq & C_k(||E||_{C^0(\overline{B_1})}+||\underline f||_{C^{k-2,\alpha}(\overline{B_1})})\\
&\leq & C_k(||E||_{C^0(\overline{B_1})}+||\underline f||_{C^{k-1}(\overline{B_1})})\\
&=&  \left\{
  \begin{array}{llll}
    O(|x|^{-\min\{n,\zeta\}}), & \zeta\not=n,\\
    O(|x|^{-n}(\ln|x|)), & \zeta=n,\\
  \end{array}
  \right.
  \text{ as }|x|\rightarrow\infty.
\end{array}
$$
This finishes the proof of Theorem \ref{Thm-firstExpansion}.

\section{Proof of Theorem \ref{Thm-secondExpansion}}

In this  section, we consider asymptotic expansion at infinity for classical solutions of \eqref{Equ-exterior}. Assume that $u,f$ are as in Theorem \ref{Thm-firstExpansion}. Let $\overline{a_{ij}}, \overline f$ and $v$ be as in \eqref{linearized-equation-3} and subsection \ref{sec-Pf-1.1} respectively.

In the following, we only need to focus on $\zeta>n$ case as explained in Remark \ref{thm-radial}.
It follows from \eqref{equ-asym-Behavior} in Theorem \ref{Thm-firstExpansion},
\begin{equation*}
\left|\overline{a_{i j}}(x)-D_{M_{i j}} F_{\tau}(\lambda(A))\right| \leq C\left|D^{2} v(x)\right|=O_{m-1}\left(|x|^{-n}\right)
\end{equation*}
and hence
$$
\begin{array}{llll}
D_{M_{i j}} F_{\tau}(\lambda(A))D_{ij}v &= & \overline f-(\overline{a_{ij}}(x)-D_{M_{i j}} F_{\tau}(\lambda(A)))D_{ij}v=:g(x)\\
&=&O_m(|x|^{-\zeta})+ O_{m-1}\left(|x|^{-2n}\right) \\
&=& O_{m-1}(|x|^{-\min\{2n,\zeta\}})
\end{array}
$$
by \eqref{Low-Regular-Condition}
as $|x|\rightarrow\infty$.

Let $$Q:=
[D_{M_{i j}} F_{\tau}(\lambda(A))]^{\frac{1}{2}}\quad\text{and}\quad \widetilde v(x):=v(Qx).
$$
Then
\begin{equation}\label{temp-1}
\Delta \widetilde v(x)=g(Qx)=:\widetilde g(x)\quad\text{in } \mathbb R^n.
\end{equation}
By a direct computation,
$$
\widetilde v=O_{m+1}(|x|^{2-n})\quad\text{and}\quad \widetilde g=O_{m-1}\left(|x|^{-\min\{2 n,\zeta\}}\right).
$$

Let $\Delta_{\mathbb{S}^{n-1}}$ be the Laplace-Beltrami operator on unit sphere $\mathbb{S}^{n-1}\subset\mathbb{R}^n$ and
$$
\Lambda_0=0,~\Lambda_1=n-1,~\Lambda_2=2n,~\cdots,~\Lambda_k=k(k+n-2),~\cdots,
$$
be the sequence of eigenvalues of $-\Delta_{\mathbb S^{n-1}}$ with eigenfunctions
\begin{equation*}Y_1^{(0)}=1,~Y_{1}^{(1)}(\theta),~Y_{2}^{(1)}(\theta),~\cdots,~ Y_{n}^{(1)}(\theta),~\cdots,~Y_{1}^{(k)}(\theta),~\cdots,~Y_{m_k}^{(k)}(\theta),~\cdots
\end{equation*}
i.e.,
$$
-\Delta_{\mathbb{S}^{n-1}}Y_m^{(k)}(\theta)=\Lambda_kY_m^{(k)}(\theta),\quad\forall~
m=1,2,\cdots,m_k.
$$

By Lemmas 3.1 and 3.2 of \cite{bao-liu-2020}, there exists a solution $\widetilde v_{\widetilde g}$ of $\Delta \widetilde v_{\widetilde g}=\widetilde g$ in $\mathbb R^n\setminus\overline{B_1}$ with
\begin{equation*}
\widetilde v_{\widetilde g}=\left\{\begin{array}{ll}
O_{m}\left(|x|^{2-\min\{2 n,\zeta\}}\right), & \min\{2 n,\zeta\}-n \notin \mathbb{N}, \\
O_{m}\left(|x|^{2-\min\{2 n,\zeta\}}(\ln |x|)\right), & \min\{2 n,\zeta\}-n \in \mathbb{N}.
\end{array}\right.
\end{equation*}
Thus
$\overline v(x):=\widetilde v-\widetilde v_{\widetilde g}$ is harmonic on $\mathbb R^n\setminus\overline{B_1}$ with $\overline v=O(|x|^{2-n})$ as $|x|\rightarrow\infty$.
By spherical harmonic expansions, there exist constants $C_{k,m}^{(1)}, C_{k,m}^{(2)}$ such that  $$
  \overline{v}=\sum_{k=0}^{\infty} \sum_{m=1}^{m_{k}} C_{k,m}^{(1)}Y_{m}^{(k)}(\theta) |x|^{k} +\sum_{k=0}^{\infty} \sum_{m=1}^{m_{k}} C_{k,m}^{(2)} Y_{m}^{(k)}(\theta) |x|^{2-n-k}.
  $$
  By the vanishing speed of $\overline v$, we have $C_{k,m}^{(1)}=0$ for all $k,m$. Thus similar to the proof of Lemma 3.3 in \cite{bao-liu-2020}, there exist constants $c_{k,m}$ with $k\in\mathbb N$, $m=1,\cdots,m_k$ such that
\begin{equation*}
\widetilde v=
\left\{
\begin{array}{llll}
\displaystyle \sum_{k=0}^{[\zeta]-n} \sum_{m=1}^{m_{k}} c_{k, m}Y_{m}^{(k)}(\theta)|x|^{2-n-k} +
O_{m}\left(|x|^{2-\zeta}\right), & n<\zeta<2n,~\zeta\not\in\mathbb N,\\
\displaystyle \sum_{k=0}^{\zeta-n-1} \sum_{m=1}^{m_{k}} c_{k, m} Y_{m}^{(k)}(\theta)|x|^{2-n-k}+
O_{m}\left(|x|^{2-\zeta}(\ln |x|)\right), & n<\zeta< 2n,~\zeta\in\mathbb N,\\
\displaystyle \sum_{k=0}^{n-1} \sum_{m=1}^{m_{k}} c_{k, m}
Y_{m}^{(k)}(\theta)|x|^{2-n-k} +
O_{m}\left(|x|^{2-2n}(\ln |x|)\right), & 2n\leq \zeta.\\
\end{array}
\right.
\end{equation*}
By rotating backwards by $Q^{-1}$,
the results in Theorem \ref{Thm-secondExpansion} follow immediately.

\small

\bibliographystyle{plain}

\bibliography{AsymExpan}

\bigskip

\noindent Z.Liu \& J. Bao

\medskip

\noindent  School of Mathematical Sciences, Beijing Normal University\\
Laboratory of Mathematics and Complex Systems, Ministry of Education\\
Beijing 100875, China \\[1mm]
Email: \textsf{liuzixiao@mail.bnu.edu.cn}\\[1mm]
Email: \textsf{jgbao@bnu.edu.cn}





\end{document}